\documentclass[preprint, 10pt]{elsarticle}
\usepackage[margin=2.5cm]{geometry}
\usepackage{amsthm}
\usepackage{amsmath}
\usepackage{amssymb}
\usepackage{mathtools}
\usepackage{natbib}
\usepackage{enumitem}

\title{Convergence in distribution of the Bernstein-Durrmeyer kernel and pointwise convergence of a generalised operator for functions of bounded variation}
\author[1]{Mohammed Taariq Mowzer}
\ead{TaariqMowzer@gmail.com}

\address[1]{Department of Mathematics and Applied Mathematics, University of Cape Town}

\newtheorem{theorem}{Theorem}[section]
\newtheorem{lemma}[theorem]{Lemma}
\newtheorem{corollary}[theorem]{Corollary}
\newtheorem{definition}[theorem]{Definition}

\newcommand{\Beta}{Beta}
\newcommand{\N}{N}
\newcommand{\Bin}{Bin}
\newcommand*{\bigchi}{\mbox{\large$\chi$}}

\begin{document}
\begin{abstract}
We study the convergence of Bernstein type operators leading to two results. The first:  The kernel $K_n$ of the Bernstein-Durrmeyer operator at each point $x \in (0, 1)$ --- that is $K_n(x, t) dt$ --- once standardised converges to the normal distribution. The second result computes the pointwise limit of a generalised Bernstein-Durrmeyer operator applied to --- possibly discontinuous --- functions $f$ of bounded variation.
\end{abstract}

\maketitle
\section{Introduction}
Bernstein operators $B_n(f)(x) =  \sum_{k = 0}^n p_{n,k}(x)f(k/n)$, where $p_{n,k}(t) = \binom{n}{k} t^k(1 - t)^{n - k}$, were introduced by S.Bernstein \cite{bernstein} to give a constructive proof of the Weierstrass approximation theorem. The Bernstein-Durrmeyer operator below was constructed by Durrmeyer \cite{durrmeyer} and independently by Lupa{\c{s}} \cite{lupas}:
\begin{definition}
    The Bernstein-Durrmeyer operator $M_n$ applied to an integrable function $f$ on $[0, 1]$ is given by
    \begin{equation*}
        M_n(f)(x) = \int_0^1 f(t)K_n(x, t) \, dt, \quad x \in [0, 1],
    \end{equation*}
    where $K_n$ is a Bernstein-Durrmeyer kernel given by
    \begin{equation*}
        K_n(x, t) = (n + 1) \sum_{k = 0}^n p_{n,k}(x)p_{n,k}(t), \quad x, t \in [0, 1].
    \end{equation*}
\end{definition}

In his original proof, Bernstein used the binomial distribution and the weak law of large numbers to constructively prove the Weierstrass approximation theorem. 

In this paper, we use the fact that the binomial distribution converges to the normal distribution as the number of trials increases to show that the kernel of the Bernstein-Durrmeyer operator tends to the normal distribution. To be precise:

\begin{theorem}\label{thrm:normal_convergence_kernel}
    For $x \in (0, 1)$ fixed
    \begin{equation*}
        \frac{1}{\sqrt{n}}K_n\left(x, \frac{t}{\sqrt{n}} + x\right) dt \text{ converges in distribution to } N(0, 2x(1- x)) \text{ as } n \rightarrow \infty,
    \end{equation*}
	where $\N(\mu, \sigma^2)$ is the normal distribution with mean $\mu$ and variance $\sigma^2$.
\end{theorem}
Equivalently, if $X_n$ is a random variable with distribution $K_n(x, t)dt$ then
    \begin{equation*}
        \sqrt{n}(X_n - x) \text{ converges in distribution to } \N(0, 2x(1- x)) \text{ as } n \rightarrow \infty.
    \end{equation*}
    
The Bernstein-Durrmeyer operator can be generalised to arbitrary measure. Here, we limit ourselves to arbitrary weights. 
\begin{definition}
    The Bernstein-Durrmeyer operator with non-negative weight $w$  is given by
    \begin{equation}\label{eqn:def_berdurr_gen}
        M_{n, w}(f)(x) = \sum_{k = 0}^n p_{n,k}(x) \frac{\int_0^1 f(t) p_{n,k}(t) w(t) \, dt}{\int_0^1 p_{n,k}(t) w(t) \, dt}, 
    \end{equation}
    where $w$ is chosen so that the terms are well defined, $x \in [0, 1]$, and $f$ is a function on $[0, 1]$ integrable with respect to $w(t)dt$.
\end{definition}
In the particular case when $w$ is a Jacobi weight $w(t) = t^\alpha(1 - t)^\beta$ for $\alpha, \beta>-1$, the operator $M_{n,w}$ is well known. For example, Berens and Xu (e.g. \cite{berens}) investigated the Bernstein-Durrmeyer operator with Jacobi weights, and Ditzian (e.g. \cite{multidimensional}) studied the multidimensional Bernstein-Durrmeyer operator with Jacobi weights. The case of arbitrary measure --- replace $w(t)dt$ with $d \rho(t)$ in (\ref{eqn:def_berdurr_gen}) --- and its multivariate analogue was, to our knowledge, for the first time studied by Berdysheva and Jetter \cite{kurt}. Berdysheva \cite{BERDYSHEVA2014734} showed that the Bernstein-Durrmeyer operator with arbitrary measure when applied to $f$ at $x$ converges to $f(x)$ if $f$ is essentially bounded, continuous at $x$ and the support of $\rho$ contains $x$. 

Herzog and Hill proved in \cite{discontbernstein} that for $f$ defined on $\mathbb{Q}\cap [0, 1]$, $f$ bounded and having left and right limits $f(x-)$ and $f(x+)$ at a given $x \in (0, 1)$, for the classical Bernstein operator
$$
\lim_{n \rightarrow \infty} B_n(f)(x) = \frac{1}{2}f(x-) + \frac{1}{2}f(x+).
$$

Guo \cite{GUO1987183} showed when $f$ is of bounded variation that 
\begin{equation}\label{eqn:Well-known-disc}
\lim_{n \rightarrow \infty} M_n(f)(x) = \frac{1}{2}f(x-) + \frac{1}{2}f(x+)
\end{equation}
holds for the Bernstein-Durrmeyer operator and gave an estimate of the rate of convergence.

In this paper we find the limit of the  Bernstein-Durrmeyer operator with arbitrary weights when applied to a function $f$ of bounded variation. To be precise, we show:
\begin{theorem}\label{thrm:result_stated_twice_first}
Let $w$ be a non-negative function of bounded variation on $[0, 1]$. Let $x \in (0, 1)$ be fixed. Assume that either the left limit $w(x-)$ or the right limit $w(x+)$ at $x$ is non-zero. For any function $f$ of bounded variation on $[0, 1]$,
\begin{equation}\label{eqn:weighted_char}
    \lim_{n \rightarrow \infty}M_{n,w}(f)(x) = (1 - \nu)f(x-) + \nu f(x+),
\end{equation} where \begin{equation*}\nu =  \int_0^1 \frac{w(x+)(1 - u)}{w(x-)u + w(x+)(1 - u)} \, du.
\end{equation*}
\end{theorem}
When $w(x-) = w(x+)$, $\nu$ evaluates to $\frac{1}{2}$. Otherwise,
\begin{equation} \nonumber
 \nu = \frac{r^2 - r[1 + \ln r]}{(r - 1)^2}, \quad \text{for } r = \frac{w(x+)}{w(x-)} \text{ when } r \in (0, \infty)\text{ and }r \neq 1.
\end{equation} 

Theorem \ref{thrm:result_stated_twice_first} is an extension of (\ref{eqn:Well-known-disc}) --- consider $w = 1$. In fact (\ref{eqn:Well-known-disc}) holds even if only $w(x+) = w(x-)$, that is, $w$ continuous at $x$.
\section{Prerequisites and useful results}
Henceforth, we denote the binomial distribution with $n$ trials and probability $p$ of success by $\Bin(n, p)$, the normal distribution with mean $\mu$ and variance $\sigma^2$ by $\N(\mu, \sigma^2)$, and the beta distribution with $\alpha$ and $\beta$ parameters by $\Beta(\alpha, \beta)$. The symbol $\stackrel{D}{\rightarrow}$ denotes convergence in distribution and $\sim$ denotes having the same distribution. We denote the probability measure by $\mathbb{P}$ and expectation of a random variable by $\mathbb{E}$.
	
The function $p_{n,k}$ is extended onto $\mathbb{R}$, and $k$ is allowed on $[0, n]$ so that
\begin{equation*}
    p_{n,k}(t) = 
    \begin{cases}\frac{\Gamma(n + 1)}{\Gamma(k + 1) \Gamma(n - k + 1)}t^k(1 - t)^{n - k} &\text{ if } t \in[0, 1],\\
    0 &\text{ if } t \not \in [0, 1].
    \end{cases}
\end{equation*} For a function $f$, $f(t-)$ and $f(t+)$ denote the left and right limit of $f$ at $t$, and $\mathcal{D}(f)$ is the set of discontinuities of $f$.  Usually $x$ will be fixed in $(0, 1)$. It is safe to define $\Psi$ as the cumulative distribution function (cdf) of $N(0, x(1 - x))$.
\begin{theorem}
\label{thrm:bin_normal_approx}
    Let $C_n$ have a $\Bin(n, p)$ distribution for fixed $p \in (0, 1)$. Then $$\sqrt{n}(C_n/n - p) \stackrel{D}{\rightarrow} N(0, p(1 - p)), \text{ as } n \rightarrow \infty.$$
\end{theorem}
\begin{theorem}
\label{thrm:beta_normal_approx}
    Let $\beta_{r_1, r_2}$ have distribution $\Beta(r_1, r_2)$. Then 
    \begin{equation}\label{eqn:beta_nice_converge}
    \sqrt{r_1 + r_2}\left(\beta_{r_1,r_2} - \frac{r_1}{r_1 + r_2}\right) \stackrel{D}{\rightarrow} N(0, \gamma(1 - \gamma)),
    \end{equation} as $r_1, r_2 \rightarrow \infty$ and $\frac{r_1}{r_1 + r_2} \rightarrow \gamma$ for $\gamma \in (0, 1)$.
\end{theorem}
Theorem \ref{thrm:bin_normal_approx} is a direct consequence of the central limit theorem and is well known. Theorem \ref{thrm:beta_normal_approx} is a generalisation of the well known approximation that the standardised $\Beta(an, bn)$ distribution tends to the normal distribution as $n \rightarrow \infty$. As far as we know, Theorem \ref{thrm:beta_normal_approx} is not well known. We supply a proof in the appendix. 

The Corollary below will be the most useful form of Theorem \ref{thrm:beta_normal_approx}.

\begin{corollary}\label{cor:beta_normal_approx}
 Suppose $B_{n, k}$ has a $\Beta(k + 1, n - k + 1)$ distribution, and let $x \in (0, 1)$ be fixed. Then
 \begin{equation}\label{eqn:beta_bad_converge}\sqrt{n}\left(B_{n, k} - x\right) \stackrel{D}{\rightarrow} N(\alpha, x(1 - x))\end{equation} as $n\rightarrow \infty$ and $\sqrt{n}\left(\frac{k}{n} - x\right) \rightarrow \alpha$. 
\end{corollary}
\begin{proof}
    Notice $\sqrt{n}\left(\frac{k}{n} - x\right) \rightarrow \alpha$ implies $\frac{k}{n} \rightarrow x$. By Theorem \ref{thrm:beta_normal_approx},
    \begin{equation*}
    {\overbrace{\frac{\sqrt{n + 2}}{\sqrt{n}}}^M}\left[\sqrt{n}\left(B_{n,k} - \frac{k}{n}  \right) \right] + {\overbrace{\sqrt{n + 2}\left(\frac{k}{n} - \frac{k + 1}{n + 2}\right)}^C} = \sqrt{n + 2}\left(B_{n, k} - \frac{k + 1}{n + 2}\right) \stackrel{D}{\rightarrow} N(0, x(1 - x))
    \end{equation*}
  as $n \rightarrow \infty$ and $\frac{k}{n} \rightarrow x$ . Notice $M \rightarrow 1$ and $C \rightarrow 0$ in the limit. Therefore, by Slutsky's Theorem, $\sqrt{n}\left(B_{n,k} - \frac{k}{n}\right) \stackrel{D}{\rightarrow} N(0, x(1 - x))$. Another application of Slutsky's Theorem gives
  $$
  \sqrt{n}(B_{n,k} - x) = \sqrt{n}\left(B_{n,k} - \frac{k}{n}\right) + \sqrt{n}\left(\frac{k}{n} - x\right) \stackrel{D}{\rightarrow} N(\alpha, x(1 - x))
  $$
  as $n \rightarrow \infty$ and $\sqrt{n}\left(\frac{k}{n} - x\right) \rightarrow \alpha$.
\end{proof}

The following gives a straightforward way to show uniform convergence on compact intervals using sequences.

\begin{lemma}\label{lem:compact_crit}\cite[Chapter~3]{remmert1991theory}.
    Let $f$ and $(f_n)$ be functions defined on a compact subset $K$ of $\mathbb{R}$. The following are equivalent
    \begin{enumerate}[label=\roman*)]
        \item $(f_n)$ converges uniformly to $f$, and $f$ is continuous on $K$.
        \item For any sequence $(\alpha_n) \subset K$, if $\lim_{n \rightarrow \infty} \alpha_n = \alpha$  then $ \lim_{n \rightarrow\infty} f_n(\alpha_n) = f(\alpha)$.
    \end{enumerate}
\end{lemma}

\begin{proof}
    See for example \cite[Chapter~3]{remmert1991theory} section on continuous convergence.
\end{proof}

The next Theorem allows us to use convergence in distribution to show convergence in expectation.

\begin{theorem}\label{thrm:bounded_convergence}
Let $(X_n)$ be a sequence of real-valued random variables which converge to the random variable $X$ in distribution. Let $(h_n)$ be a sequence of uniformly bounded measurable functions from $\mathbb{R}$ to $\mathbb{R}$ which converge uniformly on compact intervals to a measurable function $h$ whose set of discontinuities $\mathcal{D}(h)$ is closed and satisfies $\mathbb{P}(X \not \in \mathcal{D}(h)) = 1$. Then 
$$
\lim_{n \rightarrow \infty} \mathbb{E}(h_n(X_n)) = \mathbb{E}(h(X)).
$$
\end{theorem}

\begin{proof}
    By Skorokhod's representation theorem, see for example \cite[Chapter~3.5]{Erhan}, there exists a sequence of random variables $(Y_n)$ such that each $Y_n$ has the same distribution as $X_n$, and $(Y_n)$ converges almost surely to a random variable $Y$ having the same distribution as $X$. Let $\Omega_0$ be the almost sure set where $(Y_n)$ converges pointwise to $Y$ and let $\Omega_1$ be the almost sure set $\{Y \not \in \mathcal{D}(h)\}$. For $\omega$ in the almost sure set $\Omega_0 \cap \Omega_1$, $Y(\omega)$ is in the open set $\mathcal{D}(h)^c$. Hence, $h$ is continuous on a compact neighbourhood of $Y(\omega)$. By Lemma~\ref{lem:compact_crit}, uniform convergence of $h_n$ to $h$ on the compact neighbourhood yields $h_n(Y_n(\omega)) \rightarrow h(Y(\omega))$ as $n \rightarrow \infty$. Therefore, $h_n(Y_n) \rightarrow h(Y)$ almost surely. Since $(h_n)$ is uniformly bounded, the sequence $(h_n(Y_n))$ is dominated by some constant. By Lebesgue's dominated convergence theorem, $$\lim_{n \rightarrow \infty} \mathbb{E}(h_n(X_n)) = \lim_{n \rightarrow \infty} \mathbb{E}(h_n(Y_n)) = \mathbb{E}(h(Y)) = \mathbb{E}(h(X)).$$
\end{proof}

\section{Normality of Bernstein-Durrmeyer kernel}

For the rest of this section $x \in (0, 1)$ is fixed. We construct a probabilistic representation of the Bernstein-Durrmeyer kernel $K_n(x, \cdot)$. Consider the distribution generated from the samples of the following process.

\begin{enumerate}[label = Step \arabic*., leftmargin=3\parindent]
\item Obtain a coin with probability $x$ of obtaining heads.
\item Flip the coin $n$ times recording the number of heads.
\item If $k$ heads were recorded, the sample is randomly drawn from a $\Beta(k + 1, n - k + 1)$ distribution.
\end{enumerate}
Suppose $C_n\sim \Bin(n, x)$ and $B_{n,k}~\sim~\Beta(k~+~1,~n~-~k~+~1)$ form an independency of random variables. The random variable 
\begin{equation*}
    X_n = \sum_{k = 0}^n \bigchi_{\{C_n = k\}}B_{n, k}
\end{equation*}
describes samples from the above process. Here, $\bigchi_F$ is the indicator function on a set $F$ that is
$$\bigchi_F(x) = \begin{cases} 1 & \text{ if } x \in F, \\
0 & \text{ if } x \not \in F. \end{cases}
$$
We show $X_n$ has distribution $K_n(x, \cdot)dt$. First, we note the probability density function (pdf) of $B_{n,k}$ is $$f_{B_{n,k}}(t) = \frac{\Gamma(n + 2)}{\Gamma(k + 1)\Gamma(n - k + 1)}t^k(1 - t)^{n - k} = (n + 1)p_{n,k}(t) \quad \text{for } t \in[0, 1]$$
and the probability mass function (pmf) of $C_n$ is
$$
p_{C_n}(k) = \binom{n}{k} x^k(1- x)^{n - k} = p_{n,k}(x) \quad \text{ for } k \in \{0,\ldots ,n\}. 
$$
Now, for all $z \in \mathbb{R}$,
\begin{align}
    \mathbb{P}(X_n \leq z) &= \mathbb{E}[\mathbb{P}(X_n \leq z\, |\, C_n = k)]\\
    &= \mathbb{E}[\mathbb{P}(B_{n,k} \leq z \,|\, C_n = k)] \\
    &= \sum_{k = 0}^n p_{n,k}(x)\mathbb{P}(B_{n,k} \leq z)\\
    &= \sum_{k = 0}^n p_{n,k}(x)\int_0^z (n + 1)p_{n,k}(t) \, dt \\
    &= \int_0^z K_n(x, t) \, dt.
\end{align}
Independence of $C_n$ and $B_{n,k}$ is used in the third equality. Hence, $X_n$ has distribution $K_n(x, t)dt$.

Now, the transformation formula says that the distribution of $\sqrt{n}(X_n - x)$ is 
$$
\frac{1}{\sqrt{n}}K_n\left(x, \frac{t}{\sqrt{n}} + x\right)dt.
$$
Therefore we can prove Theorem \ref{thrm:normal_convergence_kernel} by showing the equivalent statement:

\begin{theorem}\label{thrm:X_goes_normal}
    Let $x \in (0, 1)$ be fixed. Suppose $X_n$ is defined as above. Then $$\sqrt{n}(X_n - x) \stackrel{D}{\rightarrow} N(0, 2x(1 - x)).$$
\end{theorem}
Both the binomial and beta distribution, when standardised, converge to the normal distribution. It makes sense that $X_n$ --- which is a mixture of a binomial distribution with beta distributions --- when standardised, also converges to the normal distribution. Indeed, Corollary \ref{cor:beta_normal_approx} says $\sqrt{n}(B_{n, k} - x)$ is approximately $N(\alpha, x(1 - x))$ where $\alpha = \sqrt{n}\left(k/n - x\right)$ for large $n$ and $k$. In the random process $X_n$, $k$ equals $C_n \sim \Bin(nx, nx(1 - x))$. Now $\alpha$ is the random variable $C^*_n = \sqrt{n}\left(C_n/n - x\right)$ which is approximately $\N(0, x(1 - x))$ for large $n$. So, intuitively, for large enough $n$
$$
\sqrt{n}(X_n - x) \sim \N(\alpha, x(1 - x)) = \alpha + \N(0, x(1 - x)) = \N(0, x(1 - x)) + \N(0, x(1 - x)) = N(0, 2x(1 - x)).
$$
Most of the challenge of the proof comes from the fact that as $|\alpha|$ increases much larger values of $n$ are required for the approximation $\sqrt{n}(B_{n, k} - x) \sim N(\alpha, x(1 - x))$ to hold. However, it turns out that large $|\alpha|$ is improbable since $\alpha$ is approximately $\N(0, x(1 - x))$. 
\begin{proof}[Proof of Theorem \ref{thrm:X_goes_normal}]
Let $z \in \mathbb{R}$ be fixed. Let  $C_n^* = \sqrt{n}(C_n/n - x)$. Then
\begin{align}
\mathbb{P}(\sqrt{n}(X_n - x) \leq z) &= \mathbb{E}[\mathbb{P}(\sqrt{n}(B_{n,k} - x) \leq z \, | \, C_n = k)] \nonumber\\
&= \mathbb{E}[\mathbb{P}(\sqrt{n}(B_{n,xn + \alpha \sqrt{n}} - x) \leq z \, | \, \sqrt{n}(C_n/n - x) = \alpha)] \nonumber\\
\label{al:hnCn}&= \mathbb{E}[\mathbb{P}(\sqrt{n}(B_{n,xn + \alpha \sqrt{n}} - x) \leq z \,| \, C_n^* = \alpha)] = \mathbb{E}(h_n(C_n^*)),
\end{align}
where the last equality follows from the independence of $C_n$ and $B_{n, k}$, and
$$
h_n(\alpha) = 
\begin{cases}
    \mathbb{P}(\sqrt{n}(B_{n, xn + \alpha \sqrt{n}} - x) \leq z) & \text{ if } xn + \alpha\sqrt{n} \in [0, n],\\
    0 & \text{ otherwise.}
\end{cases}
$$

Let $K$ be any non-empty compact subset of $\mathbb{R}$. Let $(\alpha_n) \subset K$ such that $\lim_{n \rightarrow \infty} \alpha_n = \alpha$. Then $\lim_{n \rightarrow \infty}\sqrt{n}\left(\frac{xn + \alpha_n \sqrt{n}}{n} - x\right) = \alpha$. Corollary \ref{cor:beta_normal_approx} then gives $\sqrt{n}(B_{n, xn + \alpha_n \sqrt{n}} - x) \stackrel{D}{\rightarrow} \N(\alpha, x(1 - x))$ as $n \rightarrow \infty$. Let $A$ be a random variable with distribution $N(\alpha, x(1 - x))$ and recall $\Psi$ is the cdf of $\N(0, x(1 - x))$. Convergence in distribution yields,
\begin{equation}\label{eqn:F(z-alpha)}
\lim_{n \rightarrow \infty} h_n(\alpha_n) = \lim_{n \rightarrow \infty}  \mathbb{P}(\sqrt{n}(B_{n, xn + \alpha_n \sqrt{n}} - x) \leq z) = \mathbb{P}(A \leq z) = \Psi(z - \alpha).
\end{equation}
Lemma \ref{lem:compact_crit} gives that $h_n$ converges uniformly to $\Psi(\cdot - \alpha)$ on $K$. Since $K$ was arbitrary, uniform convergence holds on all compact intervals.

Let $U$ and $V$ be independent random variables with distribution $\N(0, x(1 - x)).$ By Theorem \ref{thrm:bin_normal_approx}, $C_n^* \stackrel{D}{\rightarrow} V$ as $n \rightarrow \infty$ . Equations (\ref{al:hnCn}), (\ref{eqn:F(z-alpha)}) and Theorem \ref{thrm:bounded_convergence} give
\begin{align*}
\lim_{n \rightarrow \infty} \mathbb{P}(\sqrt{n}(X_n - x) \leq z) = \lim_{n \rightarrow \infty} \mathbb{E}(h_n(C_n^*)) = \mathbb{E}(\Psi(z - V)) &= \mathbb{E}(\mathbb{P}(U \leq z - V\,|\, V))\\ &= \mathbb{P}(U \leq z - V) = \mathbb{P}(U + V \leq z).
\end{align*}
Since $U$ and $V$ are independent, the sum $U + V$ has distribution $N(0, 2x(1 - x))$. Hence, the cdf of $\sqrt{n}(X_n - x)$ converges pointwise to the cdf of $N(0, 2x(1 - x))$.
\end{proof}
\section{Limit of the generalised Bernstein-Durrmeyer operator}
We introduce the Lupa\c{s} Beta operator \cite{lupas} for integrable functions $f$ on $[0, 1]$
$$
L_n(f)(y) = \int_0^1 f(\theta)\frac{\theta^{ny}(1 - \theta)^{n(1 - y)}}{\mathcal{B}(ny + 1, n(1 - y) + 1)}\, d\theta,
$$
where $\mathcal{B}$ is the Beta function. Notice that the kernel of the Lupa\c{s} Beta operator is the $\Beta(ny + 1, n(1 - y) + 1)$ distribution. Specifically, for $B_{n, k} \sim \Beta(k + 1, n - k + 1)$, $0 \leq k \leq n$, we have that
\begin{equation}\label{eqn:Lupas}
L_n(f)(k/n) = \mathbb{E}(f(B_{n, k})).
\end{equation}

For the rest of this section $x \in (0, 1)$ is fixed. We can express the generalised Bernstein-Durrmeyer operator as
$$
M_{n, w}(f)(x) = \sum_{k = 0}^n p_{n,k}(x) \frac{\int_0^1 f(t) p_{n,k}(t) w(t) \, dt}{\int_0^1 p_{n,k}(t) w(t) \, dt} = \sum_{k = 0}^n p_{n,k}(x)\frac{L_n(f w)(k/n)}{L_n(w)(k/n)}.
$$

Using (\ref{eqn:Lupas}),  the Lupa\c{s} Beta operators can be represented as an expected value of a beta distribution. The summation can be represented as the expected value of a binomial distribution. Using these probabilistic representations we can decompose the proof of Theorem \ref{thrm:result_stated_twice_first} into first evaluating the Lupa\c{s} Beta operator, then evaluating the binomial part. In both cases, Theorem \ref{thrm:bounded_convergence} will be used to evaluate the limit.

\begin{lemma}\label{lem:incomplete_beta_approx}
Let $f:[0, 1] \rightarrow \mathbb{R}$ be of bounded variation. Let $k_n(\alpha) = xn + \alpha \sqrt{n}$. The limit \begin{equation} \label{eqn:cnk_to_G}
\lim_{n \rightarrow \infty} L_n(f)(k_n(\alpha)/n) = \Psi(-\alpha)f(x-) + [1 - \Psi(-\alpha)]f(x+),
\end{equation} holds uniformly in $\alpha$ on all compact subsets of $\mathbb{R}$.
\end{lemma}

\begin{proof} We look to apply Lemma \ref{lem:compact_crit}. Let $K$ be a non-empty compact subset of $\mathbb{R}$. Let $(\alpha_n) \subset K$ such that $\lim_{n \rightarrow \infty} \alpha_n = \alpha$. Clearly, $\lim_{n \rightarrow \infty} \sqrt{n}\left(\frac{k_n(\alpha_n)}{n} - x\right) = \alpha$. Hence, Corollary \ref{cor:beta_normal_approx} yields

\begin{equation}\label{eqn:Bstar_def}
B_{n,k_n(\alpha_n)}^* := \sqrt{n}\left(B_{n,k_n(\alpha_n)} - x \right) \stackrel{D}{\rightarrow} N(\alpha, x(1 - x)) \text{ as } n \rightarrow \infty.
\end{equation}
Using (\ref{eqn:Lupas}), we represent the Lupa\c{s} Beta operator probablistically and get that 
\begin{equation}\label{eqn:half_int_to_Bstar}
    L_n(f)(k_n(\alpha_n)/n) = \mathbb{E}(f(B_{n,k_n(\alpha_n)})) = \mathbb{E}(h_{n}(B_{n,k_n(\alpha_n)}^*)),
\end{equation}
where $$h_n(t) = f\left(\frac{t}{\sqrt{n}} + x\right)\bigchi_{[0, 1]}\left(\frac{t}{\sqrt{n}} + x\right) =  f\left(\frac{t}{\sqrt{n}} + x\right)\bigchi_{[-x\sqrt{n}, (1 - x)\sqrt{n}]}(t),\quad t \in \mathbb{R}.$$ 

From the decomposition
$$\bigchi_{[-x\sqrt{n}, (1-x)\sqrt{n}]} = \bigchi_{[-x\sqrt{n}, 0)} + \bigchi_{\{0\}} + \bigchi_{(0, (1 - x)\sqrt{n}]},$$ we notice $h_n$ limits uniformly to $h = f(x-)\bigchi_{(-\infty, 0)} + f(x)\bigchi_{\{0\}} + f(x+)\bigchi_{(0, \infty)}$ on $K$. Since $f$ is of bounded variation, $(h_n)$ is uniformly bounded. Let $A$ be a random variable with distribution $\N(\alpha, x(1 - x))$. The set of discontinuities $\mathcal{D}(h) \subset \{0\}$ is closed, and is a null set of the distribution of $A$. Hence, equations (\ref{eqn:half_int_to_Bstar}), (\ref{eqn:Bstar_def}) and Theorem \ref{thrm:bounded_convergence} give
\begin{equation}\label{eqn:half_int_to_final}
 L_n(f)(k_n(\alpha_n)/n) = \mathbb{E}(h_n(B_{n,k_n(\alpha_n)}^*)) \stackrel{n \rightarrow \infty}{\longrightarrow} \mathbb{E}(h(A)) = \Psi(-\alpha)f(x-) + [1 - \Psi(-\alpha)]f(x+).
\end{equation}
If we consider the sequence of functions $\left(L_n(f)(k_n(\cdot)/n)\right)_n$ and the converging sequence $(\alpha_n)$, an application of Lemma \ref{lem:compact_crit} on (\ref{eqn:half_int_to_final}) gives that convergence of (\ref{eqn:cnk_to_G}) is uniform in $\alpha$ on $K$. Since $K$ was arbitrary, (\ref{eqn:cnk_to_G}) holds uniformly in $\alpha$ on all compact subsets of $\mathbb{R}$.
\end{proof}

\begin{proof}[Proof of Theorem \ref{thrm:result_stated_twice_first}]
Notice
\begin{equation}\label{eqn:M_nw_to_expected}
M_{n, w}(f)(x) = \sum_{k = 0}^n p_{n,k}(x) c_n(k) = \mathbb{E}(c_n(C_n)),
\end{equation}
where $C_n \sim \Bin(n, x)$ and $$c_n(k) = \frac{L_n(fw)(k/n)}{L_n(w)(k/n)} \text{ if } k \in [0, n], \text{ otherwise } c_n(k) = 0.$$

We transform $C_n$ and $c_n$ as follows. Let $C_n^* = \sqrt{n}\left(\frac{C_n}{n} - x\right)$ and define $c_n^*$ so that $c_n(C_n) = c_n^*(C_n^*)$, that is $c_n^*(\alpha) = c_n(xn + \alpha \sqrt{n})$ for $\alpha \in \mathbb{R}$. Then 
\begin{equation}\label{eqn:M_nw_to_hn}
    M_{n, w}(f)(x) = \mathbb{E}(c_n(C_n)) = \mathbb{E}(c_n^*(C_n^*)).
\end{equation}

We define functions $t_n, b_n$ as follows
\begin{equation}\nonumber
    t_n(\alpha) = L_n(fw)(x + \alpha/\sqrt{n}) \quad \text{and} \quad b_n(\alpha) = L_n(w)(x + \alpha/\sqrt{n})
\end{equation}
for $\alpha \in \mathbb{R}$.

By Lemma \ref{lem:incomplete_beta_approx},
\begin{equation}\label{eqn:tn-limits}
    \lim_{n \rightarrow \infty} t_n(\alpha) = f(x-)w(x-)\Psi(-\alpha) + f(x+)w(x+)(1 - \Psi(-\alpha))
\end{equation}
and
\begin{equation}\label{eqn:bn-limits}
\lim_{n \rightarrow \infty} b_n(\alpha) = w(x-)\Psi(-\alpha) +  w(x+)(1 - \Psi(-\alpha))
\end{equation}
uniformly in $\alpha$ on compact subsets of $\mathbb{R}$.
By (\ref{eqn:tn-limits}) and (\ref{eqn:bn-limits})
\begin{equation}\label{eqn:requiredlimit}
    \lim_{n \rightarrow \infty} c_n^*(\alpha) = \lim_{n \rightarrow \infty} \frac{t_n(\alpha)}{b_n(\alpha)} =  \frac{f(x-)w(x-)\Psi(-\alpha) + f(x+)w(x+)(1 - \Psi(-\alpha))}{w(x-)\Psi(-\alpha) +  w(x+)(1 - \Psi(-\alpha))}
\end{equation}
uniformly on compact subsets of $\mathbb{R}$. Notice $w(x-) \neq 0$ or $w(x+) \neq 0$ is used here to guarantee that on every compact subset of $\mathbb{R}$ the denominator $b_n$ is eventually uniformly lower bounded.

Since $f$ is of bounded variation it is bounded, so $(c_n)$ and $(c_n^*)$ are uniformly bounded. Let $V$ have a $N(0, x(1 - x))$ distribution. By Theorem \ref{thrm:bin_normal_approx}, $(C_n^*)$ converges in distribution to $V$. Equation (\ref{eqn:requiredlimit}) says that $(c_n^*)$ converges uniformly on compact intervals to the right hand side of (\ref{eqn:requiredlimit}) which we shall call $c^*$. Recalling (\ref{eqn:M_nw_to_hn}), an application of Theorem \ref{thrm:bounded_convergence} gives
\begin{equation} \label{eqn:shortcut}
   \lim_{n \rightarrow \infty} M_{n, w}(f)(x) = \lim_{n \rightarrow \infty} \mathbb{E}(c_n^*(C_n^*)) = \mathbb{E}(c^*(V)).
\end{equation}

Expanding out the right hand side gives
\begin{equation}\label{eqn:normal-bernstein-integral}
    \lim_{n \rightarrow \infty}M_{n,w}(f)(x) = \int_{-\infty}^{\infty}\frac{f(x-)w(x-)\Psi(-\alpha) + f(x+)w(x+)(1 - \Psi(-\alpha))}{w(x-)\Psi(-\alpha) +  w(x+)(1 - \Psi(-\alpha))} \, d\Psi(\alpha),
\end{equation} The substitution $u = \Psi(-\alpha)$ in (\ref{eqn:normal-bernstein-integral}) gives (\ref{eqn:weighted_char}).
\end{proof}

\section{Acknowledgements}
Much thanks to Professor E.E. Berdysheva for suggesting this problem, for greatly helping in the preparation of the paper, and for providing reading materials on the history of the Bernstein-Durrmeyer operators. The author also thanks F.C. Rath  for bringing to attention the work of Professor R.J. Ryder \cite{ryder} who the author thanks for providing further references on the normal approximation of the beta distribution. Thank you to the referees for their suggestions --- especially for drawing our attention to the Lupas Beta operator --- and careful reading of this paper.

This work is supported in part by the National Research Foundation of South Africa PMDS22103166345.


\bibliography{reference}

\appendix
\section{Proof of Theorem \ref{thrm:beta_normal_approx}}
\begin{proof}
We follow the proof method used in \cite{ryder} which shows that the $\Beta(n, n)$ goes to the normal distribution as $n \rightarrow \infty$. That is for $\beta_{r_1, r_2} \sim \Beta(r_1, r_2)$, we show that the pdf of $\sqrt{r_1 + r_2}(\beta_{r_1, r_2} - \frac{r_1}{r_1 + r_2})$ converges pointwise to the pdf of the $N(0, \gamma(1 - \gamma))$ distribution as $r_1, r_2 \rightarrow \infty$ and $\frac{r_1}{r_1 + r_2} \rightarrow \gamma$. By Scheffe's Theorem \cite{Scheffethrm}, pointwise convergence of the   pdf implies convergence in distribution.

The beta distribution has pdf $$f_{\beta_{r_1, r_2}}(x) = \frac{\Gamma(r_1 + r_2)}{\Gamma(r_1)\Gamma(r_2)}x^{r_1 - 1}(1 - x)^{r_2 - 1} \quad \text{for } 0 <  x < 1.$$

Using the transformation formula, the pdf of $Y_{r_1, r_2} = \sqrt{r_1 + r_2}(\beta_{r_1, r_2} - \frac{r_1}{r_1 + r_2})$ is
\begin{align*}
f_{Y_{r_1, r_2}}(y) &= \frac{1}{\sqrt{r_1 + r_2}}f_{\beta_{r_1, r_2}}\left( \frac{r_1}{r_1 + r_2} + \frac{y}{\sqrt{r_1 + r_2}}\right)\\
&= \frac{1}{\sqrt{r_1 + r_2}}\frac{\Gamma(r_1 + r_2)}{\Gamma(r_1)\Gamma(r_2)}\left( \frac{r_1}{r_1 + r_2} + \frac{y}{\sqrt{r_1 + r_2}}\right)^{r_1 - 1}\left(\frac{r_2}{r_1 + r_2}-\frac{y}{\sqrt{r_1 + r_2}}\right)^{r_2 - 1},
\end{align*}
for $-\frac{r_1}{\sqrt{r_1 + r_2}} < y < \frac{r_2}{\sqrt{r_1 + r_2}}$.
We show pointwise convergence in the log of the pdf. Note
\begin{align*}
    \ln (f_{Y_{r_1, r_2}}(y)) = &\underbrace{-\ln{\sqrt{r_1 + r_2}} + \ln(\Gamma(r_1 + r_2)) - \ln(\Gamma(r_1)) - \ln(\Gamma(r_2))}_A\\
    &+\underbrace{(r_1 - 1)\ln\left(\frac{r_1}{r_1 + r_2} + \frac{y}{\sqrt{r_1 + r_2}}\right) + (r_2 - 1) \ln\left({\frac{r_2}{r_1 + r_2}-\frac{y}{\sqrt{r_1 + r_2}}}\right)}_B
\end{align*}
Stirling's approximation gives $\ln(\Gamma(n)) = n \ln n - n + \ln\sqrt{2\pi/n} + O(\frac{1}{n})$, $n \rightarrow \infty$. Approximating $A$,
\begin{align*}
    A &= - \ln\sqrt{r_1 + r_2} + (r_1 + r_2)\ln(r_1 + r_2) - (r_1 + r_2) + \ln\sqrt{2\pi/(r_1 + r_2)} + O\left(\frac{1}{r_1 + r_2}\right)\\
    &\quad- \left[ r_1\ln r_1  -r_1 + \ln\sqrt{2\pi/r_1} + O\left(\frac{1}{r_1}\right)\right] - \left[r_2\ln r_2  -r_2 + \ln\sqrt{2\pi/ r_2} + O\left(\frac{1}{r_2}\right)\right]\\
    &= - \ln\sqrt{2\pi} + r_1\ln\frac{r_1 + r_2}{r_1} + r_2\ln\frac{r_1 + r_2}{r_2} +\frac{1}{2}\ln\frac{r_1}{r_1 + r_2} + \frac{1}{2}\ln\frac{r_2}{r_1 + r_2} + O\left(\frac{1}{r_1} + \frac{1}{r_2}\right).
\end{align*}
The limit is pointwise; so fix $y$. Note $\frac{r_1}{r_1 + r_2} \rightarrow \gamma$, so $y\frac{\sqrt{r_1 + r_2}}{r_1}$ and $y\frac{\sqrt{r_1 + r_2}}{r_2}$ limit to 0. Hence, we can use the second order Taylor Series approximation $\ln(1 + x) = x - \frac{x^2}{2} + O(x^3)$ as $x \rightarrow 0$. Approximating $B$,
\begin{align*}
    B &= (r_1 - 1)\left[\ln\frac{r_1}{r_1 + r_2} + \ln(1 + y\frac{\sqrt{r_1 + r_2}}{r_1})\right]\\
    &\quad+  (r_2 - 1)\left[\ln\frac{r_2}{r_1 + r_2} + \ln(1 - y\frac{\sqrt{r_1 + r_2}}{r_2})\right]\\
    &= r_1\ln\frac{r_1}{r_1 + r_2} + y\sqrt{r_1 + r_2} -\frac{1}{2} y^2\frac{r_1 + r_2}{r_1} - \ln\frac{r_1}{r_1 + r_2} + O\left(\frac{1}{\sqrt{r_1}}\right)\\
    &\quad+ r_2\ln\frac{r_2}{r_1 + r_2} - y\sqrt{r_1 + r_2} - \frac{1}{2}y^2 \frac{r_1 + r_2}{r_2} - \ln\frac{r_2}{r_1 + r_2} + O\left(\frac{1}{\sqrt{r_2}}\right).
\end{align*}

Combining approximations for $A$ and $B$ we get 
\begin{align*}
A + B &= \left[ r_1\ln\frac{r_1 + r_2}{r_1} + r_1\ln\frac{r_1}{r_1 + r_2} \right] + \left[r_2\ln\frac{r_1 + r_2}{r_2} + r_2\ln\frac{r_2}{r_1 + r_2} \right] + \left[y\sqrt{r_1 + r_2} - y\sqrt{r_1 + r_2} \right]\\
&\quad+\left[\frac{1}{2}\ln\frac{r_1}{r_1 + r_2}-\ln\frac{r_1}{r_1 + r_2}\right] + \left[\frac{1}{2}\ln\frac{r_2}{r_1 + r_2} - \ln\frac{r_2}{r_1 + r_2} \right] \\
&\quad-\ln\sqrt{2\pi} -\frac{1}{2}y^2\left(\frac{r_1 + r_2}{r_1} +\frac{r_1 + r_2}{r_2}\right)  +  O\left(\frac{1}{\sqrt{r_1}} + \frac{1}{\sqrt{r_2}}\right)\\
&= \quad- \frac{1}{2}\ln\frac{r_1}{r_1 + r_2} - \frac{1}{2}\ln\frac{r_2}{r_1 + r_2} -\ln\sqrt{2\pi} -\frac{1}{2}y^2\left(\frac{r_1 + r_2}{r_1} +\frac{r_1 + r_2}{r_2}\right) + O\left(\frac{1}{\sqrt{r_1}} + \frac{1}{\sqrt{r_2}}\right)\\
&\rightarrow -\frac{1}{2}\ln\gamma - \frac{1}{2}\ln(1-\gamma) -\ln\sqrt{2\pi} - \frac{1}{2}\frac{y^2}{\gamma(1-\gamma)} = \ln\left(\frac{1}{\sqrt{2\pi\gamma(1-\gamma)}}e^{-\frac{1}{2}\frac{y^2}{\gamma(1-\gamma)}}\right),
\end{align*}
as $r_1,r_2 \rightarrow\infty$ and $\frac{r_1}{r_1 + r_2} \rightarrow \gamma$. This is the natural log of the pdf of the $N(0, \gamma(1-\gamma))$ distribution.
\end{proof}
\end{document}